\documentclass[11pt]{amsart}

\usepackage{epsfig}
\usepackage{amsfonts}
\usepackage{amssymb}
\usepackage{amsmath}
\usepackage{amsthm}
\usepackage{amscd}
\usepackage{amstext}
\usepackage{latexsym}

\usepackage{setspace}

\newcommand{\CC}{\mathbf{C}}

\newcommand{\JJ}{\mathcal{J}}
\newcommand{\OO}{\mathcal{O}}

\newcommand{\II}{\mathcal{I}}

\newcommand{\iddb}{\sqrt{-1} \partial \overline{\partial}}

\newcommand{\db}{\overline{\partial}}
\newcommand{\al}{\alpha}

\newcommand{\om}{\Omega}

\newcommand{\ep}{\epsilon}

\newcommand{\el}{\ell}

\newcommand{\varep}{\varepsilon}
\newcommand{\qa}{\quad}
\newcommand{\vp}{\varphi}
\newcommand{\vep}{\varep}
\newcommand{\noi}{\noindent}

\newcommand{\evp}{e^{-\varphi}}

\newcommand{\zb}{\bar{z}}

\newcommand{\gp}{g_1, \cdots, g_p}

\providecommand{\abs}[1]{\left\vert #1\right \vert}

\providecommand{\norm}[1]{\lVert#1\rVert}

\providecommand{\hs}[2]{ {\mathcal H}^{#1}_{#2}}

\theoremstyle{plain}

\newtheorem{theorem}{Theorem}[section]

\newtheorem{lemma}[theorem]{Lemma}
\newtheorem{corollary}[theorem]{Corollary}
\newtheorem{proposition}[theorem]{Proposition}

\newtheorem{conjecture}[theorem]{Conjecture}

\newtheorem*{division problem}{Division Problem}

\theoremstyle{definition}
\newtheorem{definition}[theorem]{Definition}

\theoremstyle{remark}

\newtheorem{remark}[theorem]{Remark}

\DeclareMathOperator{\divisor}{div}

\DeclareMathOperator{\Ker}{Ker}
\DeclareMathOperator{\Image}{Im}

\DeclareMathOperator{\Dom}{Dom}

\DeclareMathOperator{\Real}{Re}

\DeclareMathOperator{\Skod}{Skod}
\DeclareMathOperator{\Tor}{Tor}

\long\def\symbolfootnote[#1]#2{\begingroup\def\thefootnote{\fnsymbol{footnote}}
\footnote[#1]{#2}\endgroup}

\renewcommand{\thefootnote}{}

\author{Dano Kim  }

\title{The Exactness of a general Skoda complex}

\begin{document}

\maketitle

\begin{abstract}

\noi  We show that a Skoda complex with a general plurisubharmonic weight function is exact if its `degree' is sufficiently large. This answers a question of Lazarsfeld and implies that not every integrally closed ideal is equal to a multiplier ideal even if we allow general plurisubharmonic weights for the multiplier ideal, extending the result of Lazarsfeld and Lee~\cite{LL}.

\end{abstract}

\footnote{\noindent This work was supported by the  National Research Foundation of Korea grants NRF-2012R1A1A1042764 and No.2011-0030795, funded by the Korea government and also by Research Startup Fund for new faculty of Seoul National University. } 

\section{Introduction}

 In complex algebraic geometry, a singular weight  function of the form $\frac{1}{\abs{f}^2} =: e^{-\vp}$ plays an important role, where $f$ is a holomorphic function. It is natural to consider more generally, a singular weight $e^{-\vp}$ where $\vp$ is a plurisubharmonic function. Given $e^{-\vp}$, there are two fundamental ways to define an ideal sheaf of local holomorphic function germs, say $u$  : collecting those with $\abs{u}^2 e^{-\vp}$ locally bounded above, on one hand and collecting those with local integral $\int_\Omega \abs{u}^2 e^{-\vp}$ finite, on the other hand. The former gives an \emph{integrally closed} ideal and the latter a \emph{multiplier} ideal.

 A multiplier ideal is always an integrally closed ideal, but the converse had been unknown  in general dimension until \cite{LL} showed the existence of an integrally closed ideal that is not a multiplier ideal of a psh function with analytic singularities (e.g. of the form $\log \abs{f}^2$ for $f$ holomorphic). In this paper, we extend this result to the full generality of multiplier ideals of arbitrary psh functions. 
  `ementary consequence of a local vanishing theorem \cite[(9.4.4), (9.6.36)]{L}. However, the general case of the exactness of a Skoda complex cannot be equally shown from the vanishing theorem because of the difficult openness conjecture~(\ref{openness}) which is not known beyond dimension $2$ (see \cite{FJ}). Instead of vanishing, we use  the $L^2$ methods of \cite{Sk72} to a Skoda complex setting and prove

\begin{theorem}\label{main}
Let $X$ be a complex manifold and $L$ and $M$ line bundles on $X$. Let $e^{-\psi}$ be a singular hermitian metric with psh weight for the line bundle $M$.
 Let $g_1, \cdots, g_p \in H^0(X,L)$. Then there exists an integer $q \ge p$ such that the $q$-th Skoda complex associated to $\gp$ is exact. More precisely, one can take $q = \lfloor \frac{1}{4} p^2 + \frac{1}{2} p + \frac{5}{4} \rfloor $.

\end{theorem}

 Theorem~\ref{main} can be considered as a generalized version of the Skoda type division theorem~\cite{Sk72} in that we {divide} not only at the right end of the Skoda complex but also at all the other intermediate terms of the Skoda complex.  Indeed, we separate the division statement as Proposition~\ref{effprop} which follows from the proof of (\ref{main}) (see Remark~\ref{eff}).

  The value of $q$ in (\ref{main}) comes from $\lfloor \frac{1}{4} p^2 + \frac{1}{2} p + \frac{5}{4}  \rfloor = \max_{0 \le m \le p} m(p-m+1)+1 $. We believe the optimal value of $q$ which can be used in the statement will be $q = p$ as is also indicated by (\ref{openskoda}). But the present value of $q$ is sufficient to establish the following generalization of the main result of \cite{LL}.

\begin{corollary}\label{mainc}

 There exist a complex algebraic variety $X$ and an integrally closed ideal sheaf $\mathfrak{b}$ on it such that $\mathfrak{b}$ cannot be written as a multiplier ideal sheaf $\JJ(\varphi)$ even if we allow $\varphi$ to be a general plurisubharmonic function.

\end{corollary}

\noi Note that this is generalization of the main result of \cite{LL} since a priori the class of all analytic multiplier ideal sheaves associated to a psh function might be strictly larger than the class of all algebraic multiplier ideal sheaves. (We also remark here that the note added at the end of \cite{LL} was made when we did not realize yet that the proof of the implication (\ref{openskoda}) was depending on (\ref{openness}). )

 This paper is organized as follows. In Section 1, we motivate and give the definition of a Skoda complex. In Section 2, we discuss the openness conjecture of \cite{DK} and show that it implies the exactness of a Skoda complex in full generality. Section 3, we prove our main theorem~(\ref{main}), the exactness of a Skoda complex not assuming the openness conjecture, of course. In Section 4, we follow \cite{LL} to derive the existence of an integrally closed ideal that is not a multiplier ideal~(\ref{mainc}).

\begin{remark}

 Apart from the use for local syzygy in \cite{LL}, a Skoda complex was originally used to prove the Skoda-type division theorem in an algebraic way (using cohomology vanishing) in \cite{EL}. On the other hand, the original analytic way of \cite{Sk72} to prove the Skoda-type division theorem (not via cohomology vanishing) does not involve the use of a Skoda complex. But it is interesting to note that \cite{H67} had used a Koszul complex (together with his $L^2$ methods for $\db$) for a prototype result toward Skoda division. Later it was replaced by the more refined $L^2$ methods of \cite{Sk72}.

\end{remark}

\noindent \textit{Acknowledgement.} The author would like to thank Professor Robert Lazarsfeld for informing him of the question of the exactness of a Skoda complex with general plurisubharmonic functions, at an AIM workshop in 2006. The author also thanks Professor Lawrence Ein for his suggestion and encouragement to use the methods of \cite{Sk72} in the present setting.  The author is also grateful to Kyungyong Lee for useful discussions and to Sebastien Boucksom for helpful answers to some questions on the openness conjecture.
\\

\section{Definition of a Skoda complex  }

 Let $X$ be a complex manifold and $L$ a line bundle on $X$. Let $g_1, \cdots, g_p \in H^0(X, L)$ be holomorphic sections. Let $M$ be another line bundle.  Let $\mathcal{A} (M)$ denote either the set of all holomorphic sections of $M$ or the set of all complex-valued measurable sections of $M$. Given $u \in \mathcal{A} (M)$, we can ask whether there exist $h_1, \cdots, h_p \in \mathcal{A}(M- L)$  such that $u = h_1 g_1 + \cdots + h_p g_p$. Such a division problem is concerned with the surjectivity of the multiplication map $P: \mathcal{A} (M- L)^{\oplus p} \to \mathcal{A} (M) $ given by $ (v_1, \cdots, v_p) \mapsto  v_1 g_1 + \cdots + v_p g_p $ from the direct sum of $p$ copies of $\mathcal{A}(M- L)$ on the left. In some approaches to the division problem, one needs to extend the single map $P$ to a Koszul-type complex to the left:

 $$ 0 \to  \mathcal{A} (M-pL)^{\oplus {p\choose p}} \to \cdots \to   \mathcal{A} (M-2L)^{\oplus {p\choose2}}  \to    \mathcal{A} (M- L)^{\oplus p} \to \mathcal{A} (M) \to 0    .$$

\noindent where we use the basis  $ \{ e_{i_1} \wedge \cdots \wedge e_{i_m}  \vert  1 \le i_1 < \cdots < i_m \le p  \} $  for $\mathcal{A} (M- mL)^{\oplus {p\choose m}} =: \mathcal{B}_m$ and the map $P: \mathcal{B}_m \to \mathcal{B}_{m-1}$ in the complex is the usual Koszul map

\begin{equation}\label{koszul}
 P ( e_{i_1} \wedge \cdots \wedge e_{i_m} ) = \sum_{k=1}^m (-1)^{k-1} g_{i_k} e_{i_1} \wedge \cdots \wedge \widehat{e_{i_k}} \wedge \cdots \wedge e_{i_m}   .
\end{equation}

\noindent $P$ is defined by \eqref{koszul} and the linearity in $\mathcal{A}$. In effect, we introduced $e_i$'s in order to define the map $P$. Now it is also natural to consider the weight $\varphi = \log \abs{g}^2$ (defining $\abs{g}^2 := \abs{g_1}^2 + \cdots + \abs{g_p}^2$ throughout this paper) and (for $1 \le m \le p$) the subset $\mathcal{A} (M - mL , e^{- (p -m) \vp } e^{-\psi}   ) \subset \mathcal{A} (M - mL) $ of sections that are square-integrable with respect to $ e^{- (p -m) \vp } e^{-\psi}  dV$ where $dV$ is a volume form and $e^{-\psi}$ is a (auxiliary) weight for the line bundle $(M- pL)$.

\begin{proposition}
\qa

 The restriction of the Koszul map $P$ to  $\mathcal{A} (M - mL , e^{- (p -m) \vp} e^{-\psi} )^{\oplus{ p \choose m} } $   has its image contained in  $\mathcal{A} (M - (m-1)L , e^{- (p - m + 1) \vp} e^{-\psi} )^{\oplus{p \choose m-1} }   $.

\end{proposition}

\begin{proof}

Let $u \in \mathcal{A} (M - mL , e^{- (p -m) \vp} e^{-\psi}  )^{\oplus{ p \choose m} } $ and write it as $u = \sum_J u_J e_J$ where  the index $J$ denotes $(j_1, \cdots, j_m)$ with $1 \le j_1 < \cdots < j_m \le p$.   We know that for each index $J$, $\int_X \abs{u_J}^2  e^{- (p -m) \vp}  e^{-\psi}  dV  < \infty  $. Consider $P(u)$ and its $I$-th component where the index $I$ denotes $(i_1, \cdots, i_{m-1})$ with $1 \le i_1 < \cdots < i_{m-1} \le p$. Call the $I$-th component as $\sigma$. Then

$$ \sigma = \sum_{t \notin I, 1 \le t \le p} g_t u_{I \cup t} $$

\noi where the index $I \cup t$ of $u_{I \cup t}$ denotes the rearrangement in the right order, $|I \cup t|$ being $m$.  Now the conclusion follows from Cauchy-Schwarz:

 $$   \abs{\sum_{t \notin I, 1 \le t \le p} g_t u_{I \cup t} }^2   e^{- (p- (m - 1)) \vp} e^{-\psi}  \le  \abs{g}^2 \sum_J \abs{u_{J} }^2  e^{- (p-(m - 1)) \vp} e^{-\psi}. $$
 
\end{proof}

\noindent Consequently we have the following Koszul-type complex:

\begin{multline*}
 0 \to  \mathcal{A} (M-pL, e^{-\psi})^{\oplus {p\choose p}} \to   \mathcal{A} (M-(p-1)L, e^{-(\psi + \vp)})^{\oplus {p\choose p-1}}  \to  \cdots  \\
  \to  \mathcal{A} (M- L,  e^{-(\psi + (p-1)\vp)} )^{\oplus p} \to  \mathcal{A} (M,  e^{-(\psi + p\vp)} ) \to 0
\end{multline*}

 Now if we restrict our attention to holomorphic sections and the corresponding sheaves, we get a complex of coherent sheaves:

\begin{definition}[{\cite{EL}, \cite[(9.6.36)]{L}}]\label{skoplex}

  Let $X$ be a complex manifold and $L$,$M$ line bundles on $X$. Let $(M,e^{-\psi})$ be a singular hermitian metric. Let $V$ be the vector space spanned by holomorphic sections $ g_1, \cdots, g_p \in H^0(X, L)$. Let $q \ge p$.  The $q$-th \textbf{Skoda complex} $(\Skod_q)$ is the complex of Koszul maps (which we will construct in the below)

\begin{multline}\label{skocom}
0 \rightarrow \Lambda^p V \otimes \mathcal{J} \left( (q-p) \varphi + \psi \right) \otimes \OO((q- p)L + M)
  \rightarrow \cdots    \\
  \cdots \rightarrow \Lambda^m V \otimes \mathcal{J} ( (q-m)\varphi + \psi) \otimes \OO((q-m)L + M)   \rightarrow \cdots  \\
  \cdots \rightarrow  \Lambda^1 V \otimes \mathcal{J} (  (q-1)\varphi + \psi) \otimes \mathcal{O} ((q-1)L + M)  \rightarrow \mathcal{J} (   q \varphi + \psi) \otimes \OO(qL + M) \rightarrow 0
\end{multline}

\noindent where we recall that $\varphi = \log \abs{g}^2 = \log (\abs{g_1}^2 + \cdots + \abs{g_p}^2)$. 

\end{definition}

\begin{proof}[Construction of \eqref{skocom}]

 Let $ f \colon Y \rightarrow X $ be a log-resolution of the ideal $\mathfrak{a} \subset \OO_X$ generated by $g_1, \cdots, g_p$ by Hironaka's theorem. Later in (\ref{openskoda}), we will use the fact that $f$ is given by composition of blow-ups along smooth subvarieties.  Let $F$ be the exceptional divisor on $Y$ such that $ {\mathfrak a} \cdot \mathcal{O}_Y  = \mathcal{O}_Y (-F) $.  Consider the Koszul complex defined by pullbacks of generators of $\mathfrak{a}$ where $V$ is the vector space   spanned by the pullbacks.

 $$ 0 \rightarrow \Lambda^p V \otimes \mathcal{O}_Y ( pF) \rightarrow \cdots \rightarrow \Lambda^2 V \otimes \mathcal{O}_Y ( 2F) \rightarrow V \otimes \mathcal{O}_Y (F) \rightarrow \mathcal{O}_Y \rightarrow 0  $$

 Then twist through by a coherent sheaf $ O_Y ( K_{Y/ X} - qF ) \otimes \mathcal{J} (f^* \psi) $ and it stays exact since  \cite[p.7, footnote2]{EP} says that: the Koszul complex is locally split and its syzygies are locally free, so twisting by any coherent sheaf preserves exactness.    We get our Skoda complex by pushforwarding this exact sequence under $f$  since (for $0 \le m \le p$)

$$ f_* ( \mathcal{O}_Y ( K_{Y/ X} -(q-m)F) \otimes \mathcal{J} (Y, f^* \psi) ) = \mathcal{J} ( \Omega, (q-m)\varphi + \psi )  .$$ This is from the change of variables formula \cite[(9.3.43)]{L} and the fact that  $ \mathcal{O}_Y (-(q-m)F) \otimes \mathcal{J}(f^* \psi ) = \mathcal{J} ( f^* ((q-m)\varphi + \psi )) $ which in turn comes from comparing  the two sides using holomorphic function germs satisfying the local integrability conditions of the multiplier ideal sheaves.

\end{proof}

\noindent  When $\vp$ has analytic singularities, the complex $(\Skod_q)$ is exact for all $q \ge p$ by Theorem 9.6.36~\cite{L}.  In the general case, we will prove that it is exact for sufficiently large $q \ge p$.

\section{Openness conjecture for plurisubharmonic functions}

 Suppose that a plurisubharmonic function $\evp$ is given on a complex manifold $X$. Let $0 \le c < d$ be real numbers. If $(\evp)^d$ is $L^1$, then $(\evp)^c$ is $L^1$ as well since $ (\evp)^c \le (\evp)^d$.  So (fixing any compact set $K \subset X$) the set $T := \{ c \ge 0 \vert e^{-2c \vp}  \text{ is } L^1 \text{ on a neighborhood of } K  \} $ is an interval. We call $\sup T$ the \textbf{singularity exponent} of $\vp$ and write $c_K (\vp) := \sup T$. Is the interval $T$ open at the right end? The openness conjecture of \cite{DK} says so, that is, $c_K (\vp) \notin T$. On the other hand, the following statement in terms of multiplier ideal sheaves is also natural to consider:

\begin{conjecture}\label{openness}

 Let $\al$ and $\beta$ be plurisubharmonic functions on a complex manifold. Let $\JJ_+ (\al, \beta)$ be the maximal element of $\{ \JJ(\al + t \beta) \; \vert \; t > 0 \}$ as $t \to 0$. Then $\JJ_+ (\al, \beta) = \JJ(\al)$.

\end{conjecture}

\noindent The special case $\beta = \al$ gives the conjecture $\JJ_+ (\al) = \JJ (\al)$ (where $\JJ_+ (\al) := \JJ_+ (\al, \al)$) which was considered in \cite[(15.2.2)]{D} and \cite{DEL}. This special case of (\ref{openness}) implies the openness conjecture of \cite{DK} though the converse is not known. \footnote{Boucksom informed the author that it might be shown that the special case $\alpha = \beta$ of (\ref{openness}) implies the general cases using the methods of \cite{BFJ}. }

\begin{proposition}
 (\ref{openness}) implies the openness conjecture of \cite{DK}.
\end{proposition}

\begin{proof}

 Let $c = c_K (\vp)$ the singularity exponent. Take $\al = c \vp$ and $\beta = \vp$. Suppose that $c$ belongs to the interval $T$. Then $e^{-2 c \vp}$ is $L^1$, so $\JJ(c \vp)$ is trivial while for any $t > 0$, we have $\JJ(c \vp + t \vp)$ nontrivial. This contradicts  (\ref{openness}).

\end{proof}

\begin{proposition}

 (\ref{openness}) is true if $\alpha$ has analytic singularities (see \cite[Definition 1.10]{D}).

\end{proposition}

\begin{proof}

 The special case $\al = 0$ is a result of Skoda (\cite{Sk72}, see \cite[Lemma (5.6)]{D}), which we will use.  We need to show that $\JJ(\al) \subset \JJ(\al + \delta \beta)$ for some $\delta > 0$. Let $f$ be a holomorphic function germ which belongs to $\JJ(\alpha)$. Then $\exp (\log \abs{f}^2 - \alpha)$ is locally integrable. Since $\al$ has analytic singularities, we can use a log-resolution of $\alpha$ and $\divisor (f)$ to have $\ep > 0$ such that $\exp ( (1+\ep) (\log \abs{f}^2 - \al) )$ is still integrable.
 Now choose $p$ such that $\frac{1}{1+\ep} + \frac{1}{p} = 1$. Then we can choose $\delta > 0$ such that $(e^{-\beta})^{\delta p}$ is integrable from the special case $\al = 0$ of Skoda:   \cite[(5.6)~\textbf{a}]{D} says that the multiplier ideal sheaf of $\delta p \beta$ is trivial when the Lelong number of $\delta p \beta $ is less than $1$. Then it gives the finiteness of the first factor on the right in the following  H\"older inequality:

$$ \int_{\Omega} \abs{f}^2 e^{-(\al + \delta \beta)} dV \le \left( \int_{\Omega} e^{-\delta p \beta}     dV \right)^{\frac{1}{p}}  \left(\int_{\Omega} \abs{f}^{2(1+\ep)} e^{- (1+\ep)\al}  dV \right)^{\frac{1}{1+\ep}}  < \infty  .$$

\end{proof}

\noindent On the other hand, the special case of $\beta$ having analytic singularities does not seem to make (\ref{openness}) easier, as in the above way of using H\"older inequality. 

 Now we show that in fact the deep Conjecture~\ref{openness} implies the exactness of a Skoda complex via vanishing. Note that this is completely independent of our main result Theorem~\ref{main} where the exactness of a Skoda complex is proved without assuming (\ref{openness}). It seems that the methods in the proof of (\ref{openskoda}) might be useful elsewhere as well.

\begin{proposition}\label{openskoda}

If Conjecture~\ref{openness} is true, then the $q$-th Skoda complex \eqref{skocom} is exact for any $q \ge p$.

\end{proposition}

 For this, we will use Demailly's version of Nadel vanishing theorem {\cite[(5.11)]{D}} for a \textit{weakly pseudoconvex} K\"ahler manifold: a weakly pseudoconvex manifold is a complex manifold that possesses a smooth plurisubharmonic exhaustion function $\varphi$. For example, a compact complex manifold is weakly pseudoconvex, taking $\varphi = 0$. Also Stein manifolds are weakly pseudoconvex.


\begin{proof}[Proof of (\ref{openskoda})]

 Going back to the Construction of (\ref{skoplex}), first note that the log-resolution $Y$ is weakly pseudoconvex since it has the pullback under $f$ of a smooth plurisubharmonic exhaustion function on $\Omega$ as its own such exhaustion function. 
 
  We need to show the vanishing of the higher direct images:  $ R^i f_* ( O_Y ( K_{Y/ \Omega} \otimes \mathcal{J} (Y, f^*( (q-m)\varphi + \psi)) $ for $ i \ge 1 $.  Using projection formula and the fact that $\Omega$ is Stein, it suffices to show the vanishing of $ H^i (Y,  O_Y ( K_{Y/ \Omega} + f^* L ) \otimes \mathcal{J} (f^*( (q-m)\varphi + \psi)) )  $ for a sufficiently positive line bundle $L$ on $\Omega$, $i \ge 1$ and $0 \le j \le p$.  Taking $L-K_{\Omega}$ positive enough, it suffices to show that

\begin{equation}\label{vanish}
  H^i (Y,  O_Y ( K_{Y} + f^* L) \otimes \mathcal{J} (f^*( (q-m)\varphi + \psi)) ) = 0
\end{equation}

\noi  For this, we take $ F = f^*(L)$. To apply the above mentioned Nadel vanishing theorem {\cite[(5.11)]{D}}, we need to construct a singular metric $h$ of $  f^*(L) $ such that its curvature current is strictly positive and the multiplier ideal sheaf $ \mathcal{J} (Y, h) $ is the same as $ \mathcal{J} (f^*( (q-m)\varphi+ \psi)) $. This will be shown possible by giving $h$ as product of a smooth metric of a positive line bundle (which is $f^* L $ minus $E$, sum of small multiples of exceptional divisors to be specified below), the singular metric precisely given by $E$ for the $\mathbf{Q}$-line bundle $\mathcal{O} (E)$ and $f^*( (q-m)\varphi + \psi)$ ( a plurisubharmonic function, which can be seen as a singular metric of $\mathcal{O}_Y$). We use the following 

\begin{lemma} \cite[Proposition 3.24]{V}\label{voi}

  Let $f_1: Y_1 \rightarrow Y_0 $ be the blow-up of $Y_0$ along a complex submanifold $Z_0$ of codimension $k_0$ and $E_1$ be the exceptional divisor of $f_1$ in $Y_1$. Let $ A_0 $ be any ample line bundle on $Y_0$. Then there is a large enough integer $a > 1 $ such that the $\mathbf{Q}$-line bundle $ f^* A_0 - \frac{1}{a} E_1 $ is positive.

\end{lemma}

 Now let us suppose that the log-resolution $f$ is composed of smooth blow-ups $ f_M \circ f_{M-1} \circ \cdots \circ f_1 $.  Set $Y_0 := \Omega $ and $ A_0 := L $.  By abuse of notation, $E_m$ denote all the proper transforms of the exceptional divisor $E_m$ in $Y_m$ up to $Y_M$. We choose large enough integers $a_1 , \cdots , a_M$ as follows.

\begin{itemize}

\item
 $ a_1 $ is chosen to be large enough to satisfy that $ f_1^* L - \frac{1}{a_1} E_1 $ is positive by (\ref{voi}).

\item

 $ a_2 $ is chosen to be large enough so that :
 
\begin{align*} 
 f_2^* f_1^* L &= f_2^* (\underbrace{( f_1^* L - \frac{1}{a_1} E_1 )}_{positive} + \frac{1}{a_1} E_1 )    \\
    &=\underbrace{f_2^* ( (f_1^* L - \frac{1}{a_1} E_1 ) - \frac{1}{a_2} E_2}_{positive} + \frac{1}{a_2} E_2 + f_2^* ( \frac{1}{a_1} E_1 ) 
\end{align*}

\item

 $ a_3 $ is chosen to be large enough so that :
$$ f_3^* f_2^* f_1^* L = \underbrace{f_3^*  ( \cdots) - \frac{1}{a_3} E_3}_{positive} + \frac{1}{a_3} E_3 + f_3^* ( \frac{1}{a_2} E_2 + f_2^* ( \frac{1}{a_1} E_1 ))    $$  where the last three terms are rewritten as
 $ \frac{1}{a_3} E_3 + \frac{1}{a_2} f_3^* E_2 + \frac{1}{a_1} f_3^* f_2^* E_1 $.

\item

 Similarly for $a_m$ ($m \ge 4$):

 $ f_4^*f_3^* f_2^* f_1^*  L = (\text{a positive line bundle}) + {\frac{1}{a_4} E_4} + \frac{1}{a_3} f_4^* E_3 + \frac{1}{a_2} f_4^* f_3^* E_2 + \frac{1}{a_1} f_4^* f_3^* f_2^* E_1 $.

 $\qa \vdots $

 $ f^*L = f_M^* \cdots f_1^* L = (\text{positive}) + \frac{1}{a_M} E_M + \frac{1}{a_{M-1}} f_M^* E_{M-1} + \cdots + \frac{1}{a_1} f_M^* f_{M-1}^* \cdots f_2^* E_1 $.

\end{itemize}

\noindent It is now clear that we can take $E$ to be $$E := \frac{1}{a_M} E_M + \frac{1}{a_{M-1}} f_M^* E_{M-1} + \cdots + \frac{1}{a_1} f_M^* f_{M-1}^* \cdots f_2^* E_1 $$ for large enough integers $ a_1 , \cdots , a_M$ so that $ \mathcal{J} ( \varphi_{E} + f^*( (q-m)\varphi + \psi)  ) = \mathcal{J} (f^*( (q-m)\varphi + \psi)) $ ( $\varphi_{E}$ is the weight function associated to the $\mathbf{Q}$-divisor $E$) according to Conjecture~\ref{openness}. This proves (\ref{openskoda}).

\end{proof}

\section{Proof of Main Theorem}

\subsection{Algebraic preliminaries}

 Let $A$ be a commutative ring and $M$ be the dual of the free module $M' := A^{\oplus p}$ of rank $p$. We view an element of $\bigwedge^k M$ as an alternating function on $(v_1, \cdots, v_k)$ where $v_i \in  A^{\oplus p}$.  Let $\vep_1, \cdots, \vep_p$ be the basis of $M'$ and let $e_1, \cdots, e_p$  be the dual basis of $M$.  Let $h \in M'$. Let $i(h)$ be the contraction by $h$, that is (for each $m \ge 1$), the map $i(h): \bigwedge^m M \to \bigwedge^{m-1} M$ determined by

 $$ \left(i(h) \left( \eta \right) \right) (v_1, \cdots, v_{m-1}) =   \eta (h, v_1, \cdots, v_{m-1})   $$

\noindent for every $m$-form $\eta \in \bigwedge^m M$. Then it is well-known that

\begin{proposition}\label{contraction}
 For every $l, n \ge 1$ and $\vp \in \bigwedge^l M,  \psi \in \bigwedge^n M $, we have
 $$ i(h) (\varphi \wedge \psi) = (i(h) \varphi) \wedge \psi + (-1)^l \varphi \wedge ( i(h) \psi) .$$

\end{proposition}

\noindent Now taking $h = g_1 \vep_1 + \cdots + g_p \vep_p $, it is easy to see that the map $i(h):\bigwedge^m M \to \bigwedge^{m-1} M $ is our Koszul map $P$ of \eqref{koszul}. Also we let $P^{\vee}$ denote the map $e(\psi): \bigwedge^{m-1} M \to \bigwedge^{m} M$ given by taking wedge with $\frac{1}{\abs{g}^2} \psi$ where  $\psi = \overline{g_1} e_1 + \cdots + \overline{g_p} e_p$. That is,  for each $u \in \bigwedge^{m-1} M$, we have

\begin{align}
 \notag P^\vee (u) &= \frac{1}{\abs{g}^2} u \wedge  \psi  =  \frac{1}{\abs{g}^2}  ( \sum_I  u_I e_I ) \wedge   (  \overline{g_1} e_1 + \cdots + \overline{g_p} e_p   )  \\
\label{early} &= \frac{1}{\abs{g}^2} \sum_J  \sum_{k=1}^m  (-1)^{m-k} \overline{g_{j_k}} u_{ j_1 \cdots \widehat{j_k} \cdots j_m}
\end{align}

\noi where the index $J$ in the first summation denotes and ranges over $J = (j_1, \cdots, j_m)$ where $1 \le j_1 < \cdots < j_m \le p$. The last equality comes from the following argument: $e_I \wedge e_\el$ is not zero for $\el$ such that $\{ i_1, \cdots, i_{m-1}, \el \}$ has $m$ elements. Rewriting the set $\{ i_1, \cdots, i_{m-1}, \el \}$ in the increasing order as $\{ j_1, \cdots, j_m \}$ where $1 \le j_1 < \cdots < j_m \le p$, we note that

 $$ e_{i_1} \wedge \cdots \wedge e_{i_{m-1}} \wedge e_\el = (-1)^{m- k} e_{j_1} \wedge \cdots \wedge e_{j_m} $$ when $\el = j_k$ for some $1 \le k \le m$.

\noindent As a consequence of (\ref{contraction}), we have (taking $l = m-1$)

\begin{corollary}\label{comp}
\qa
\begin{enumerate}
\item
$ P ( P^{\vee} u ) = P^{\vee} ( Pu ) + (-1)^{m-1}  u  $ for all $u \in \bigwedge^{m-1} M$.

\item

 $  P ( P^{\vee} u )  =  (-1)^{m-1} u $ if $Pu = 0$.

\end{enumerate}

\end{corollary}

\noi since the map $ i(h) \psi :  \bigwedge^m M \to \bigwedge^{m} M$ is the multiplication by  $1$.
\\

Now we turn to define our Hilbert spaces and their double complex. For $i \ge 0$ (though we actually need $i = 0,1,2$ only) and $0 \le m \le p$, let $\widetilde{ \hs{m}{i} }$ be the Hilbert space completion of the smooth $((q- m)L + M)$-valued $(0,i)$ forms that are square-integrable with respect to  $e^{- (q -m) \vp} e^{-\psi}$. Let $\hs{m}{i}$ be the direct sum of $p \choose m$ copies of $\widetilde{ \hs{m}{i} }$ for which we use the basis $\{ e_{i_1} \wedge \cdots \wedge e_{i_{m}} | 1 \le i_1 < \cdots < i_m \le p  \} $. Each $\hs{m}{i}$ is given the inner product of the direct sum Hilbert space. Let $T: \hs{m}{0} \to \hs{m}{1}$ and $S: \hs{m}{1} \to \hs{m}{2}$ be the direct sum of $\db$ operators.

\begin{align}\label{dcom}
    \begin{CD}
\cdots @>>> \hs{2}{0}  @>P>> \hs{1}{0} @>P>> \hs{0}{0}   @>>> 0   \\
     @VTVV          @VTVV          @VTVV           @VTVV            \\
\cdots @>>> \hs{2}{1}  @>P>> \hs{1}{1} @>P>> \hs{0}{1}   @>>> 0   \\
     @VSVV          @VSVV          @VSVV           @VSVV            \\
\cdots @>>> \hs{2}{2}  @>P>> \hs{1}{2} @>P>> \hs{0}{2}   @>>> 0   \\
     @V{\db}VV          @V{\db}VV          @V{\db}VV           @V{\db}VV    \\
\vdots @>>> \vdots @>>> \vdots @>>> \vdots @>>>  0
        \end{CD}
\end{align}

\noi Let each $P$ be the Koszul map of \eqref{koszul}. First we compute $P^* u$ using the fact that $(P^* u, v) = (u, Pv)$ for all $v$.
\\

\noi \textbf{Throughout the rest of this paper}, the index $I$ denotes $(i_1, \cdots, i_{m-1})$ where $1 \le i_1 < \cdots < i_{m-1} \le p$ and the index $J$ denotes $(j_1, \cdots, j_m)$ where $1 \le j_1 < \cdots < j_m \le p$.

\begin{proposition}\label{adjoint}

If $P u = 0$, then $\norm{P^* u}^2 = \norm{u}^2$.

\end{proposition}

\begin{proof}

Let $u = \sum_I u_I e_I$ and $v = \sum_J v_J e_J$. Then $$P(v)= \sum_J v_J P(e_J) = \sum_J v_J \sum_{k=1}^{m} (-1)^{k-1} g_{j_k} e_{j_1} \wedge \cdots \wedge \widehat{ e_{j_k} } \wedge \cdots \wedge e_{j_m} .$$

\noi From the condition that $(P^* u, v) = (u, Pv)$ for all $v \in \hs{m}{0}$, we can determine $P^* u$. Namely, the coefficient for $v_J$ in the summation $(u, Pv)$ must be the coefficient for $e_J$ in $P^* u$. Also we take into account the fact that $(P^* u, v)$ is in $\hs{m}{0}$ and $(u, Pv)$ is in $\hs{m-1}{0}$ with one less power of $\frac{1}{\abs{g}^2}$ in the weight for the inner product.    Thus we have

\begin{equation}\label{padjoint}
P^* u = \frac{1}{\abs{g}^2} \sum_J   x_{j_1 \cdots j_m} e_{j_1} \wedge \cdots \wedge e_{j_m}  
\end{equation}

\noi where $x_{j_1 \cdots j_m} =  \sum_{k=1}^m   u_{ j_1 \cdots \widehat{j_k} \cdots j_m} (-1)^{k-1} \overline{g_{j_k}}$. From \eqref{early}, we have $ P^* u = (-1)^{m-1} P^{\vee} u$. Then $\norm{P^* u}^2 = (P^* u , P^* u) = (u, P (P^* u)) =  $

\noi  $ (u, (-1)^{m-1} P ( P^{\vee} u))  = (u,u)  $ by (\ref{comp}).



\end{proof}

\subsection{Some analytic preliminaries}



We recall the following fundamental lemmas in the methods of $L^2$ estimates for $\db$ as in \cite{Sk72}. 

\begin{lemma}[\cite{Sk72}~Proposition 1, see also \cite{Va} (3.2)]\label{FA}

 Let $\mathcal{E}_0,  \mathcal{F}_0, \mathcal{F}_1, \mathcal{F}_2$ be Hilbert spaces. Let $P: \mathcal{F}_0 \to \mathcal{E}_0$ be a bounded operator. Let $T: \mathcal{F}_0 \to \mathcal{F}_1$ and $S: \mathcal{F}_1 \to \mathcal{F}_2$ be unbounded, closely defined operators such that $S \circ T = 0 $. Let $\mathcal{G} \subset \mathcal{E}_0$ be a closed subspace such that $P (\Ker T) \subset \mathcal{G}$. We have $P (\Ker T) = \mathcal{G}$ if and only if there exists a constant $C > 0$ such that

$$\norm{P^*u + T^* \beta}^2 + \norm{S \beta}^2   \ge C \norm{u}^2 $$

\noindent for all $u \in \mathcal{G}$ and all $\beta \in \Dom T^* \cap \Dom S \subset  \mathcal{F}_1$. Moreover, in this case, for every $u \in \mathcal{G}$, there exists $v \in \Ker T$ such that $P v = u $ and $\displaystyle \norm{v} \le \frac{1}{\sqrt{C}} \norm{u}$.

\end{lemma}

\noi All the norms in the above are taken in the respective Hilbert spaces.

  Another important ingredient of the $L^2$ estimates for $\db$ is the following Bochner-Kodaira inequality, also known as the \emph{basic estimate}. Let $\Omega \subset \CC^n$ be a Stein bounded open subset and $L$ a line bundle on $\Omega$. Let $e^{-\psi}$ be a singular hermitian metric of $L$. Let $\mathcal{F}_i$ be the Hilbert space of $L$-valued $(0,i)$ forms that are square integrable with respect to $e^{-\psi}$. Let $T: \mathcal{F}_0 \to \mathcal{F}_1$ and $S: \mathcal{F}_1 \to \mathcal{F}_2$ be the $\db$ operators.

\begin{lemma}[Bochner-Kodaira \cite{H65}]\label{BK}

For all $\beta \in \Dom T^* \cap \Dom S$, we have

$$\norm{T^* \beta}^2 + \norm{S \beta}^2  \ge  \int_\om ( \iddb \psi ) (\beta, \beta) e^{-\psi}   $$

\end{lemma}

\noi Here we define $( \iddb \psi ) (\beta, \beta)$ to be (for $ \beta = \beta^1  d \zb_1 + \cdots + \beta^n d \zb_n $)    $$( \iddb \psi ) (\beta, \beta) := \sum_{1 \le p,q \le n} \frac{\partial^2 \psi}{\partial z_p \partial \overline{z}_q }  \beta^p \overline{\beta^q}    $$

\noi and regard $\int_{\Omega} ( \iddb \psi ) (\beta, \beta) e^{-\psi}$ as the norm of $\beta$ with respect to $\iddb \psi$.

\begin{remark}
In the use of H\"ormander's $L^2$ estimates for $\db$ with these lemmas, we need the standard procedure of regularizing the plurisubharmonic weight $\psi$ by a sequence of smooth plurisubharmonic functions $(\psi_\nu)_{\nu \ge 1}$. For simplicity in notations, here and in the next section, we adopt the convention that each $\psi$ means the $\nu$-th regularized $\psi_\nu$ so that we can take $\iddb \psi_\nu$ and so on. The resulting holomorphic function $u_\nu$ at each step comes with a uniform bound that is independent of $\nu$, so we can take the limit $u$ as $\nu \to \infty$ in the usual way. 
\end{remark}

\subsection{Proof of (\ref{main})}

 In the $q$-th Skoda complex \eqref{skocom}, let $S_m$ denote the sheaf  $\Lambda^m V \otimes \mathcal{J} ( (q-m)\varphi + \psi) \otimes \OO( (q- m)L + M)$ ($0 \le m \le p$).  We want to show the exactness in the middle of $S_{m+1} \stackrel{P}{\rightarrow} S_m \stackrel{P}{\rightarrow} S_{m-1}$ for every $m \ge 0$ (defining $S_{-1} := 0$).  Since the exactness of a complex is a local property, it is sufficient to show (see (\ref{effprop}) for a stronger statement with estimates)
 
\begin{equation}\label{goal}
 \Image P|_{\Omega} = \Ker P|_{\Omega} \text{\;\;\; in \;} H^0 (\Omega, S_m)  
\end{equation} 
 
\noi where $\Omega \subset X$ is a Stein open subset.

 To apply the functional analysis lemma (\ref{FA}) for this, we consider the corresponding Hilbert spaces on $\Omega$ in \eqref{dcom} and the Koszul maps $P_{m-1}: \hs{m-1}{0} \to \hs{m-2}{0}$, $P_m: \hs{m}{0} \to \hs{m-1}{0}$. In the setting of (\ref{FA}), we take $P := P_m$,  $\mathcal{E}_0 := \hs{m-1}{0}$,  $\mathcal{F}_i := \hs{m}{i}  \;\; (i = 0,1,2)$ and take $\mathcal{G} := \Ker P_{m-1} $ in $\hs{m-1}{0}$.  It suffices to show that $P_{m} (\Ker T) = \mathcal{G}$. Now consider

\begin{align}
 \norm{P^*u + T^* \beta}^2 + \norm{S \beta}^2  =&  \norm{T^* \beta}^2 + \norm{S \beta}^2 + \norm{P^*u}^2 + 2 \Real  (P^* u , T^* \beta)   .
\end{align}

\noi First, note that $\norm{P^*u}^2 = \norm{u}^2$ from (\ref{adjoint}). Our plan toward having  $C \norm{u}^2$ as in (\ref{FA}) is to divide $2 \Real  (P^* u , T^* \beta)$ into a $u$ part and a $\beta$ part. Then the $\beta$ part being less than $\norm{T^* \beta}^2 + \norm{S \beta}^2$ will finish the proof. More precisely, now apply the Bochner-Kodaira inequality~(\ref{BK}) to get $$\norm{T^* \beta}^2 + \norm{S \beta}^2  \ge \int_\Omega \sum_J   \left( (q-m)\iddb \log \abs{g}^2 + \iddb \psi \right) (\beta_J, \beta_J) e^{-\vp_1} dV      $$ where $dV$ is the Lebesgue volume form and $\vp_1 :=  (q-m) \log \abs{g}^2 + \psi $ is the weight of the Hilbert spaces $\hs{m}{i}$.   Here we have $(q-m)$ times the norm of $\beta$ with respect to $\iddb \log \abs{g}^2$. This will be cancelled out by another multiple of the same norm of $\beta$ coming out of  $2 \Real  (P^* u , T^* \beta) $. More precisely, we will show that $2 \Real  (P^* u , T^* \beta) \ge -\frac{1}{B} \norm{u}^2 - T_1 $ where $B$ is a constant $B > 1$ which we fix throughout and (see \eqref{T1})

$$T_1 := B \int_\om \abs{g}^{-2} \sum_{i_1 < \cdots < i_{m-1} } \abs{ e^{\vp} T_{i_1 \cdots i_{m-1}}}^2  e^{-\vp_1} d\lambda $$ is the second term of the RHS of \eqref{T1}. Now our main inequality to show is $$  T_1 \le( m(p-(m-1)) + 1) \int_\Omega  \sum_J       (\iddb \log \abs{g}^2) (\beta_J, \beta_J)e^{-\vp_1} dV   .$$  Then we can take $$q = \max_{0 \le m \le p} m(p-m+1)+1   $$ (which gives $q$ in (\ref{main})) so that $(\Skod_q)$ is exact, where we apply (\ref{FA}) with $C = 1 - \frac{1}{B}$.

\qa
\\

Now we begin the main computations following the above outline. In order to consider $2 \Real  (P^* u , T^* \beta) $, we write arbitrary $u \in \hs{m-1}{0}, \beta \in \hs{m}{1}$ as

\begin{align*}
       u &=  \sum_{i_1 < \cdots < i_{m-1} } u_{i_1 \cdots i_{m-1}} e_{i_1} \wedge \cdots \wedge e_{i_{m-1}}     \\
     \beta &= \sum_{j_1 < \cdots < j_m} \left( \beta^1_{j_1 \cdots j_m} d\zb_1 + \cdots + \beta^n_{j_1 \cdots j_m} d\zb_n \right) e_{j_1} \wedge \cdots \wedge e_{j_{m}},
\end{align*}

\noi we use \eqref{padjoint} to obtain  (letting $\vp := \log \abs{g}^2$)

\begin{align*}
 &2 \Real (P^* u , T^* \beta) =2 \Real ( \db (P^* u), \beta)   =    2 \Real \int_U  \sum_{j_1 < \cdots < j_m}    \mathcal{S}_J   e^{-\vp_1} dV   \text{\quad where }   \\
 \mathcal{S}_J &:= u_{j_2 \cdots j_m}  \left( \overline{ \frac{\partial}{\partial z_1} (g_{j_1} e^{-\vp} )  \beta^1_{j_1 \cdots j_m}  +
     \cdots  +
     \frac{\partial}{\partial z_n} (g_{j_1} e^{-\vp} )   \beta^n_{j_1 \cdots j_m}    }   \right)  \\
   &+ u_{j_1 j_3 \cdots j_m} \left( \overline{ \frac{\partial}{\partial z_1} (g_{j_2} e^{-\vp} )  \beta^1_{j_1 \cdots j_m}  +
     \cdots  +
     \frac{\partial}{\partial z_n} (g_{j_2} e^{-\vp} )   \beta^n_{j_1 \cdots j_m}  }    \right)    \\
   &+ \quad \cdots   \cdots  \\
   &+ u_{j_1 \cdots j_{m-1}} \left( \overline{ \frac{\partial}{\partial z_1} (g_{j_m} e^{-\vp} )  \beta^1_{j_1 \cdots j_m}  +
     \cdots  +
     \frac{\partial}{\partial z_n} (g_{j_m} e^{-\vp} )   \beta^n_{j_1 \cdots j_m}     }   \right) .
\end{align*}

\noi    Then we can rewrite this sum over $J$ as the sum over $I$ :

 $$ 2 \Real ( \db (P^* u), \beta)  = \int_\Omega \sum_{i_1 < \cdots < i_{m-1} }  u_{i_1 \cdots i_{m-1}} \overline{ T_{i_1 \cdots i_{m-1}}  }  e^{-\vp_1} dV   $$

\noi where (understanding that the index $I \cup t$ of $\beta_{I \cup t}$ denotes the rearrangement in the right order as far as $\abs{I \cup t} = m$) we define

\begin{align}
     \overline{ T_{i_1 \cdots i_{m-1}}  } :=  \sum_{t \notin I, 1 \le t \le p}  \left( \overline{ \frac{\partial}{\partial z_1} (g_{t} e^{-\vp} )  \beta^1_{I \cup t}  +
     \cdots  +
         \frac{\partial}{\partial z_n} (g_{t} e^{-\vp} )   \beta^n_{I \cup t}  }    \right)
\end{align}

\noi    Remembering that  $ \abs{u}^2 :=   \sum_{i_1 < \cdots < i_{m-1} } \abs{ u_{i_1 \cdots i_{m-1}}  }^2    $, we have

\begin{multline}\label{T1}
 \qa \qa  2 \Real ( \db (P^* u), \beta)  \ge  -\frac{1}{B} \int_\om \abs{g}^2 \abs{u}^2 e^{-2\vp - \vp_1} dV     \\   -B \int_\om \abs{g}^{-2} \sum_{i_1 < \cdots < i_{m-1} } \abs{ e^{\vp} T_{i_1 \cdots i_{m-1}}}^2  e^{-\vp_1} dV   \qa \qa \qa \qa
\end{multline}

\noi where we used the fact that for any complex numbers $X,Y$: $\abs{X}^2 + \abs{Y}^2 + 2 \Real (XY) \ge 0$    and also $\frac{1}{B} \abs{X}^2 + B \abs{Y}^2 + 2 \Real (XY) \ge 0$ for any $B \ge 1$. Then we consider

\begin{align*}
 &\abs{ e^{\vp} T_{i_1 \cdots i_{m-1}}}^2   \\
&= \abs{ e^{\vp}  \sum_{t \notin I, 1 \le t \le p}  \left( \overline{ \frac{\partial}{\partial z_1} (g_{t} e^{-\vp} )  \beta^1_{I \cup t}  +  \cdots  +   \frac{\partial}{\partial z_n} (g_{t} e^{-\vp} )   \beta^n_{I \cup t}  }    \right)    }^2   \\
&= \abs{   \sum_{t \notin I, 1 \le t \le p}  \left(  e^{\vp} \frac{\partial}{\partial z_1} (g_{t} e^{-\vp} )  \beta^1_{I \cup t}  + \cdots  +   e^{\vp}  \frac{\partial}{\partial z_n} (g_{t} e^{-\vp} )   \beta^n_{I \cup t}      \right)        }^2     \\  \label{iran}
 &= \abs{g}^{-4}    \abs{  \sum_{t \notin I, 1 \le t \le p}    \sum^n_{k=1}    \sum^p_{s=1} \overline{g_s} \left( g_s  \frac{\partial g_t }{\partial z_k} - g_t \frac{\partial g_s}{\partial z_k} \right)  \beta^k_{I \cup t}     }^2
\end{align*}

\noi noting that  (for each $t$)
\begin{align}
    e^{\vp}     \frac{\partial}{\partial z_k} (g_{t} e^{-\vp} ) = \abs{g}^{-2} \sum^p_{s=1} \overline{g_s} \left( g_s  \frac{\partial g_t }{\partial z_k} - g_t \frac{\partial g_s}{\partial z_k} \right)  .
\end{align}

\noi Finally we have the following inequalities:

{\allowdisplaybreaks

\begin{align}
 \notag  &    B  \abs{g}^{-2} \sum_{i_1 < \cdots < i_{m-1} } \abs{ e^{\vp} T_{i_1 \cdots i_{m-1}}}^2  \\
\notag   &= B  \abs{g}^{-2} \sum_{i_1 < \cdots < i_{m-1} }    \abs{g}^{-4}  \abs{  \sum_{t \notin I, 1 \le t \le p}    \sum^n_{k=1}    \sum^p_{s=1} \overline{g_s} \left( g_s  \frac{\partial g_t }{\partial z_k} - g_t \frac{\partial g_s}{\partial z_k} \right)  \beta^k_{I \cup t}     }^2    \\
\label{a1} &= B  \abs{g}^{-6} \sum_{i_1 < \cdots < i_{m-1} }  \abs{\sum^p_{s=1}  \overline{g_s}  \sum_{\substack{t \notin I \\ 1 \le t \le p}}    \sum^n_{k=1}     \left( g_s  \frac{\partial g_t }{\partial z_k} - g_t \frac{\partial g_s}{\partial z_k} \right)  \beta^k_{I \cup t}     }^2   \\
\label{a2} &\le      B  \abs{g}^{-6} \sum_{i_1 < \cdots < i_{m-1} }  \abs{g}^2  \abs{ \sum_{s=1}^p  \sum_{\substack{t \notin I \\ 1 \le t \le p}}  [[s,t]]_{I \cup t}   }^2     \\
\label{a3} &\le  (p-(m-1))   B  \abs{g}^{-4}   \sum_{i_1 < \cdots < i_{m-1} } \sum_{\substack{t \notin I \\ 1 \le t \le p}}  \abs{ \sum_{s=1}^p    [[s,t]]_{I \cup t}   }^2     \\
\label{a4}  &\le m  (p-(m-1))    B  \abs{g}^{-4} \sum_J   \sum_{1 \le q < r \le p} \abs{\;  [[q,r]]_J \;   }^2     \\
\label{a5}  &\le ( m(p-(m-1)) + 1) \sum_J       (\iddb \log \abs{g}^2) (\beta_J, \beta_J)
\end{align}
}

\noi Here we defined (from \eqref{a2} on)   $$  \left[[r,s] \right]_{J}  :=  \sum^n_{k=1}     \left( g_r  \frac{\partial g_s }{\partial z_k} - g_s \frac{\partial g_r}{\partial z_k} \right)  \beta^k_{J}  $$ for any $1 \le r, s \le p$ and the index $J$ with $\abs{J} = m$. The implications $\eqref{a1} \to \eqref{a2}$ and $\eqref{a2} \to \eqref{a3}$
are given by Cauchy-Schwarz while $\eqref{a3} \to \eqref{a4}$ follows from an elementary counting argument. Also we used that

\begin{align*}
(\iddb \log \abs{g}^2) (\beta_J, \beta_J)  &=   \sum_{1 \le p,q \le n} \frac{\partial^2}{\partial z_p \partial \overline{z}_q } ( \log \abs{g}^2 ) \beta^p_J \overline{\beta^q_J}     \\
&= \abs{g}^{-4} \sum_{1 \le q < r \le p} \abs{  \sum^n_{k=1} \left(g_q \frac{\partial g_r}{\partial z_k} - g_r \frac{\partial g_q}{\partial z_k} \right)    \beta^k_J }^2 .
\end{align*}

\noi This completes the proof of (\ref{main}).

\begin{remark}

The value of $q$ in Theorem~\ref{main} can be improved if we have  a better inequality between \eqref{a1} and $ \sum_J       (\iddb \log \abs{g}^2) (\beta_J, \beta_J)$ in  \eqref{a5}. For example, suppose that $p=2$: (\ref{main}) says that $q = \lfloor \frac{1}{4} p^2 + \frac{1}{2} p + \frac{5}{4} \rfloor = 3$ works. However, easy computation shows that \eqref{a1} is less than $2$ times $ \sum_J       (\iddb \log \abs{g}^2) (\beta_J, \beta_J)$ in this case. Therefore (the optimal value) $q=2$ also works in this case in (\ref{main}).

\end{remark}

\begin{remark}\label{eff}

 We point out that the above method of proof yields the following proposition, which is the equality~\eqref{goal} together with $L^2$ estimates. This generalizes the original Skoda division theorem in the setting of a Skoda complex. Indeed, it is completely standard in H\"ormander's $L^2$ estimates that the use of functional analysis lemma similar to (\ref{FA}) automatically produces a solution of $\db$ together with $L^2$ estimates. The only difference of the following proposition from \eqref{goal} is that we now use the last sentence of (\ref{FA}). 
 
 We use the setting and notation around \eqref{dcom} of Section 4 in the following: 
 
\begin{proposition}\label{effprop}
 Let $X = \Omega$ be a Stein manifold. Let $p$ and $q$ be integers as in (\ref{main}).   For every $m$ such that $1 \le m \le p$, let $u$ be an element of the direct sum space $\mathcal{H}^{m-1}_0$ such that each component of $u$ is holomorphic. If the norm of $u$ in $\mathcal{H}^{m-1}_0$  is finite, then there exists $v \in \mathcal{H}^{m}_0$  such that $P v = u$ and each component of $v$ is holomorphic. Moreover, there exists a constant $C > 0$ such that $\norm{v} \le C \norm{u}$ where $C$ is independent of $u$. 

\end{proposition}

It is also not hard to formulate a version for $X$ a projective manifold, similarly to the Skoda division theorem for such $X$.

\end{remark}

\section{Applications to the local syzygy}

As we discussed in the introduction, \cite{LL} used the exactness of a Skoda complex to show that not every integrally closed ideal is a multiplier ideal. Recall the setting from the introduction: let  $\Omega \subset X$ be a connected open subset of a complex manifold $X$. Let $e^{-\vp}$ be a singular weight on $\Omega$ where $\vp$ is a plurisubharmonic function on $\Omega$. From $e^{-\vp}$, there are two fundamental ways to define an ideal sheaf of local holomorphic function germs $u$  : collecting those with $\abs{u}^2 e^{-\vp}$ bounded above locally, on one hand and collecting those with the integral $\int_\Omega \abs{u}^2 e^{-\vp}$ finite, on the other hand. Let us denote the former by $\II(\vp)$ and the latter by $\JJ(\vp)$.    If $\vp$ has analytic singularities and of the form $\vp = \log ( \abs{f_1}^2 + \cdots + \abs{f_p}^2)$, then $\II(\vp)$ is the integral closure of the ideal generated by $f_1, \cdots, f_p$. It is important to distinguish $\II(\vp)$ and $\JJ(\vp)$ clearly, so let us call $\II(\vp)$ as the \textbf{sublevel} ideal sheaf of $\vp$ (while $\JJ(\vp)$ is well known as the \emph{multiplier} ideal sheaf of $\vp$). The following is essentially contained in \cite{D}.

\begin{proposition}

 A sublevel ideal sheaf $\II(\vp)$ is integrally closed.

\end{proposition}

\begin{proof}
 Suppose that a local holomorphic function $f$ satisfies an equation

 $$ f^k + a_{1} f^{k-1} + \cdots + a_{k-1} s + f_k = 0 $$

\noi where $a_i \in \II(\vp)^i$. We have the following elementary bound~\cite[Ch.II Lemma 4.10]{D97} for the roots of a monic polynomial

 $$ |f| \le 2 \max_{1 \le i \le k} \abs{a_i}^{\frac{1}{i}} $$

\noi locally. Therefore $\abs{f}^2 e^{-\vp} $ is locally bounded above.

\end{proof}

Now we apply our main theorem in the local setting as in \cite{LL}. We only need a slight modification of the statements from \cite{LL} due to the fact that our $q$ in (\ref{main}) is not optimal as in the algebraic case of (\ref{main}).  Let $X$ be a smooth complex algebraic variety of dimension $n$ and let $(\OO, \mathfrak{m})$ be the local ring of a point $x \in X$. Let $h_1, \cdots, h_p \in \mathfrak{m}$ be any collection of non-zero elements generating an ideal $\mathfrak{a} \subset \OO$. Our main theorem (\ref{main}) implies the following version of Theorem B in \cite{LL} for which we just note that our exact Skoda complex sits inbetween the two Koszul complexes in the statement.

\begin{theorem}[see \textbf{Theorem B}~\cite{LL}]
 Let $\JJ(\psi)$ be the multiplier ideal sheaf of a plurisubharmonic function $\psi$ which is defined in a neighborhood of $x$. There exists an integer $p' \ge p$ such that for every $0 \le r \le p$, the natural map

 $$ H_r \left( K_\bullet(h_1, \cdots, h_p) \otimes \mathfrak{a}^{p' -r} \JJ(\psi) \right)  \rightarrow H_r \left( K_\bullet (h_1, \cdots, h_p) \otimes \JJ (\psi) \right)   $$ vanishes.

\end{theorem}

\noindent   The original Theorem B was stronger when $\vp$ is algebraic, saying that we could actually take $p' = p$ in that case.  Now using the isomorphism between $H_r$ and $\Tor_r$ and taking $p=n$, we have

\begin{corollary}[see \textbf{Corollary C}~\cite{LL}]\label{newC}
 There exists an integer $n' \ge n$ such that the natural maps

 $$ \Tor_r ( \mathfrak{m}^{n' -r} \JJ(\psi) , \CC) \rightarrow \Tor_r (\JJ(\psi), \CC)   $$  vanish for all $0 \le r \le n$.

\end{corollary}

\begin{corollary}[see \textbf{Theorem A}~\cite{LL}]\label{newA}
 Let $\JJ = \JJ(\psi) \subset \OO$ be (the germ at $x$ of) any multiplier ideal. Then there exists an integer $n' \ge n$ with the property:  For $p \ge 1$, no minimal $p${th} syzygy of $\JJ$ vanishes modulo $\mathfrak{m}^{n' + 1 - p}$.

\end{corollary}

\begin{proof}[(\ref{newC}) implies (\ref{newA})]

 Take $n'$ from (\ref{newC}). Suppose that a minimal $p$th syzygy of $\JJ$ vanishes modulo $\mathfrak{m}^{n' + 1 -p}$.  That is, given a minimal free resolution, a linear combination of the columns of $u_p$, namely $u_p (e)$ for some $e \in R_p = \OO^{b_p} $  satisfies  $u_p (e) \in \mathfrak{m}^a R_{p-1} = \mathfrak{m}^a \cdot \OO^{b_{p-1}} $ where $a = n' + 1 - p \ge 2$. Then Proposition 1.1~\cite{LL} says that $e$ represents a class lying in the image of $\Tor_p (\mathfrak{m}^{a-1} \II, \CC) \to \Tor_p (\II, \CC) $. This contradicts to (\ref{newC}).

\end{proof}

\noi Finally,  \cite[Example 2.2]{LL}  says that there exists an integrally closed ideal $\II$ supported at a point with a first syzygy vanishing to arbitrary order $a$ at the origin. It cannot be a multiplier ideal due to (\ref{newA}). Hence Corollary~\ref{mainc} is proved since it is sufficient to examine at the local ring of a point.
\\

\noindent \textit{Note added.}  The referee kindly informed the author of recently published papers \cite{J1} and \cite{J2} (formerly arXiv:1102.3950,  arXiv:1105.4474) which contain independently obtained results similar to ours in this paper (formerly arXiv:1007.0551) : division in the Koszul complex setting (\ref{main}), (\ref{effprop}). We note that \cite[Corollary 5.7]{J1} gives an improved lower bound of $q$ from (\ref{main}) and (\ref{effprop}). Its method seems useful toward obtaining the optimal value of $q$. The author is grateful to the referee for informing him of  \cite{J1} and \cite{J2}.

\footnotesize

\bibliographystyle{amsplain}

\qa

\qa

\normalsize

\noi \textsc{Dano Kim}

\noi Department of Mathematical Sciences, Seoul National University

\noi 1 Kwanak-ro, Kwanak-gu, Seoul, Korea 151-747

\noi Email address: kimdano@snu.ac.kr

\noi

\end{document}